\documentclass{article}
\usepackage{amsmath}
\usepackage{amsfonts, amssymb, amsthm}
\usepackage{epsfig}
\usepackage{color}

\newlength{\hchng}
\newlength{\vchng}
\newtheorem{thm}{Theorem}[section]
\newtheorem{prop}[thm]{Proposition}

\newtheorem{lemma}[thm]{Lemma}

\newtheorem{preremark}[thm]{Remark}

\numberwithin{equation}{section}

\newcommand{\norm}[1]{\left\Vert#1\right\Vert}
\newcommand{\abs}[1]{\left\vert#1\right\vert}

\newcommand{\R}{\mathbb R}
\newcommand{\eps}{\varepsilon}

\newcommand{\grad} {\nabla}
\newcommand{\lap} {\triangle}

\newcommand{\dx} {\; \mathrm{d} x}
\newcommand{\dd} {\; \mathrm{d}}

\DeclareMathOperator*{\osc}{osc}

\DeclareMathOperator{\dv}{div}

\newcommand{\fl}[1] {(-\Delta)^{#1/2}}
\newcommand{\Q}{Q}

\title{Eventual regularization for the slightly supercritical quasi-geostrophic equation}
\author{Luis Silvestre}

\begin{document}
\maketitle

\begin{abstract}
We prove that weak solutions of a slightly supercritical quasi-geostrophic equation become smooth for large time. We prove it using a De Giorgi type argument using ideas from a recent paper by Caffarelli and Vasseur.
\end{abstract}


\section{Introduction}
We consider the quasi-geostrophic equation for a function $\theta : \R^2 \times [0,+\infty) \to \R$,
\begin{equation} \label{e:equation}
\begin{aligned}
\partial_t \theta(x,t) + w \cdot \grad \theta(x,t) + \fl{\alpha} \theta(x,t) &= 0\\
\theta(x,0) &= \theta_0(x)
\end{aligned}
\end{equation}

Where $\fl{\alpha} \theta = \Lambda^{\alpha} \theta$ is the fractional laplacian in the $x$ variable and $w = (-R_2 \theta, R_1 \theta) = R^\bot \theta$ where $R_i$ are the Riesz transforms
\[ R_i \theta(x) = c \mathrm{PV} \int_{\R^2} \frac{(y_i-x_i)\theta(y)}{|y-x|^3} \dd y \]

When $\alpha>1$ it is said that the equation is subcritical, and it is well known \cite{MR1709781} that solutions are smooth. In the critical case $\alpha=1$, smoothness of the solutions has been proved recently in \cite{CV2007} and \cite{kiselev2007gwp}. 

Even though both in \cite{kiselev2007gwp} and \cite{CV2007}, they obtain the global well posedness of the critical quasi-geostrophic equation, a closer look at the results and proofs reveals that they are quite different in nature. The proof in \cite{kiselev2007gwp} is certainly simpler than the one in \cite{CV2007}. The result in \cite{kiselev2007gwp} says that certain cleverly constructed modulus of continuity are preserved by the flow of the equation. In \cite{CV2007} a regularization technique inspired by De Giorgi's methods for elliptic PDEs is used to exploit the regularization effect of the equation. Thus even with $L^2$ initial data, the methods in \cite{CS2007} show that the solutions become immediately smooth.

In \cite{CV2007} the full structure of the nonlinearity in \eqref{e:equation} is not used. Their result is somewhat more general. The purpose of this paper is to use the methods of \cite{CV2007} exploiting the exact structure of the nonlinear term in \eqref{e:equation} and obtain a regularity result for the slightly supercritical case. The idea is to iteratively show that the oscillation of the function $\theta$ improves as we look at smaller parabolic cylinders, and use that information to get better local estimates for the nonlinear term $w \cdot \grad \theta$. As it is standard, this improvement of oscillation in smaller cylinders leads to a H\"older continuity result. In order to compensate for the nonlocal dependence of $w$ with respect to $\theta$, we need to make a change of variables in each iterative step that \emph{follows the flow} of the nonlocal contribution. Unfortunately this procedure works only at points $(x,t)$ if $t$ is not too small. So our result is not an immediate regularization, but instead an eventual regularization. More precisely, we prove that if $\alpha = 1-\eps$ with $\eps \ll 1$, then for any initial data $\theta_0$, there is a time $t_0$ after which the solution $\theta$ becomes smooth. This has been well known for critical QG equations for some years \cite{MR2084005} and also for many other equations (for instance Navier-Stokes), but up to our knowledge it is new for the supercritical quasi-geostrophic equation.

Our main results are
\begin{thm} \label{t:main}
Let $\theta$ be a solution of the quasi-geostrophic equation \eqref{e:equation} with initial data $\theta_0$ in $L^2$. Assume that $\alpha=1-\eps$ with $\eps \leq \delta$.
Then for any $T>0$, $\theta$ is $\delta$-H\"older continuous at time $T$ if $\delta$ is small enough. Moreover, there is an estimate
\begin{align*}
|\theta(x,T) - \theta(y,T)| \leq C |x-y|^\delta
\end{align*}
where $C$ and $\delta$ depend on $\norm{\theta_0}_{L^2}$ and $T$. 
\end{thm}

\begin{thm} \label{t:main2}
If $\eps$ is small enough, for any $\theta_0 \in L^2(\R^2)$, there is a $T_0$ such that the solution $\theta$ of \eqref{e:equation} is $C^\infty$ for $t>T_0$ ($T_0$ depends only on $\eps$ and $\norm{\theta_0}_{L^2}$, and $T_0 \to 0$ as $\eps \to 0$)
\end{thm}

The most common way to prove eventual regularity for some equation is by combining a global regularity result for small initial data with an appropriate decay of the weak solution with respect to time. We point out that our proof of Theorem \ref{t:main2} is essentially different. Even though the decay of the $L^\infty$ norm is used in the proof, after the $L^\infty$ norm is under control we still need to wait an extra period of time to obtain regularity. Our proof is not based on a perturbative argument of the critical case either.

By a \emph{solution} of \eqref{e:equation}, we mean a weak solution $\theta$ (a solution in the sense of distributions) for which the following level set energy inequality holds
\begin{equation} \label{e:energyinequality}
\int_{\R^n} \theta_\lambda^2(x,t_2) \dx + 2 \int_{t_1}^{t_2} \norm{\theta_\lambda}_{\dot H^\alpha}^2 \dd t \leq \int_{\R^n} \theta_\lambda^2(x,t_1) \dx .
\end{equation} 
where $\theta_\lambda = (\theta - \lambda)_+$ and $\norm{.}_{\dot H^\alpha}$ stands for the homogeneous Sobolev norm 
\[ \norm{f}_{\dot H^\alpha}^2 = \int |\hat f(\xi)|^2 |\xi|^{2 \alpha} \dd \xi = c \iint \frac{|f(x)-f(y)|^2}{|x-y|^{2+2\alpha}} \dd x \dd y .\]
It can be shown that such solutions exist for any initial data $\theta_0$ by adding a vanishing viscosity term $\nu \lap \theta$ to the right hand side and making $\nu \to 0$ (See \cite{MR2084005} and also the appendix in \cite{CV2007}).

The methods in this paper do not require essentially the dimension to be $2$. The same result would hold if $\theta : \R^n \to \R$ and $w = T\theta$ for some singular integral operator $T$ of order zero such that $T \theta$ is divergence free and the kernel associated to $T$ is differentiable away from the origin.

\section{Preliminaries}
In this section we review some results and constructions which are mostly adaptations from \cite{CV2007}.

\subsection*{$L^2$ and $L^\infty$ estimates.}

\begin{thm} \label{t:l2decreasing}
If $\theta$ is a solution of \eqref{e:equation} then $\norm{\theta}_{L^2}$ is decreasing in time. More precisely
\[ \norm{\theta(.,t)}_{L^2(\R^2)} \leq \norm{\theta_0}_{L^2(\R^2)} \]
\end{thm}

The theorem above is well known and could be derived directly from the energy inequality.

The following interesting theorem is an adaptation of a result from \cite{CV2007}.

\begin{thm} \label{t:l2tolinf}
If $\theta$ is a solution of \eqref{e:equation} then
\[ \sup_{x \in \R^2} |\theta(x,t)| \leq C t^{-\frac{1}{\alpha}} \norm{\theta(x,0)}_{L^2} \]
\end{thm}

The proof of Theorem \ref{t:l2tolinf} relies only on the energy inequality \eqref{e:energyinequality}. The proof of Theorem \ref{t:l2tolinf} was given in \cite{CV2007}. It was written for the case $\alpha = 1$, but the proof is general.

\subsection*{The extension problem}
It is useful to define the fractional laplacian $\fl{\alpha}$ using the the extension to the upper half space as in \cite{CS2007}. Given the function $\theta(x,t)$, we extend it to a new variable $z$ to obtain the unique function (that we still call $\theta$) $\theta(x,z,t)$ satisfying the equation
\[ \dv z^\eps \grad \theta = 0 \qquad \text{where } z > 0 \]
where $\grad \theta$ refers to the gradient in the variables $x$ and $z$. It can be proved that $(-\lap)^{\frac{1-\eps} 2} \theta(x,0,t) = \lim_{z \to 0} z^\eps \partial_z \theta(x,z,t)$. Given this construction it is now convenient to rewrite equation \eqref{e:equation} for $\alpha = 1-\eps$ in terms of the new coordinates
\begin{align}
\dv z^\eps \grad \theta &= 0 \qquad \text{where } z > 0 \label{e:qg1} \\
\partial_t \theta(x,0,t) + w \cdot \grad \theta(x,0,t) +  \lim_{z \to 0} z^\eps \partial_z \theta(x,z,t) &= 0 \label{e:qg2}
\end{align} 
the practical advantage with respect to \eqref{e:equation} is that we replaced a nonlocal operator $\fl{\alpha}$ by a local equation (in one more variable). We still have however the nonlocal contribution from $w = (-R_2 \theta,R_1 \theta)$.

We abuse notation by writing $\theta(x,t) = \theta(x,0,t)$.

\subsection*{Normalized problem}
Theorem \ref{t:l2tolinf} tells us that after any small period of time $t_0$, the solution will be in $L^\infty$. So we can assume that we have a solution in $L^\infty$ from the beginning by considering $\theta(x,t+t_0)$.

Moreover, we can rescale the function $\theta$ and consider
\[ \tilde \theta = \frac{1}{\norm{\theta}_{L^\infty}} \theta(T^{-1/\alpha} x, T^{-1} t) \]
so that we reduce the problem to the case $\norm{\theta}_{L^\infty} = 1$ and $T=1$. Including the extension variable $z$, the scaling is $\tilde \theta = \frac{1}{\norm{\theta}_{L^\infty}} \theta(T^{-1/\alpha} x, T^{-1/\alpha} z, T^{-1} \ t)$. However we will have to replace the equation (\ref{e:qg1}-\ref{e:qg2}) by
\begin{align}
\dv z^\eps \grad \theta &= 0 \qquad \text{where } z > 0 \label{e:normalized1} \\
\partial_t \theta(x,0,t) + M w \cdot \grad \theta(x,0,t) +  \lim_{z \to 0} z^\eps \partial_z \theta(x,z,t) &= 0 \label{e:normalized2}
\end{align} 
where $M$ is some constant depending only on $\norm{\theta}_{L^\infty} $ and $T$.

\subsection*{Scaling}

We use the same notation as in \cite{CV2007} appropriately scaled in terms of $\alpha$. We denote
\begin{align*}
B_r &= \{ x \in \R^2 : |x| < r \} \\
B_r^* &= B_r \times [0,r) = \{ (x,z) \in \R^3: |x| < r \ \wedge 0 \leq z < r \} \\
\Q_r &= B_r \times [0,r) \times (1-r^\alpha,1] = \{ (x,z,t) \in \R^4: |x| < r \ \text{and} \  0 \leq z < r \ \text{and} \ 1-r^\alpha < t \leq 1  \}
\end{align*}

The natural scaling of the equation is given by the fact that if $\theta$ solves (\ref{e:normalized1}-\ref{e:normalized2}), then also does $\tilde \theta(x,z,t) = \lambda^{-\eps} \theta(x_0 + \lambda x, \lambda z, t_0 + \lambda^{\alpha} t)$ for any $\lambda>0$.

On the other hand, we will use $C^\delta$ scaling, which does not preserve the equation exactly. If $\theta$ solves (\ref{e:normalized1}-\ref{e:normalized2}), then $\tilde \theta(x,z,t) = \lambda^{-\delta} \theta(x_0 + \lambda x, \lambda z, t_0 + \lambda^\alpha t)$ solves
\begin{align*}
\dv z^\eps \grad \theta &= 0 \qquad \text{where } z > 0 \\
\partial_t \theta(x,0,t) + \lambda^{\delta-\eps} M w \cdot \grad \theta(x,0,t) +  \lim_{z \to 0} z^\eps \partial_z \theta(x,z,t) &= 0
\end{align*}

Note that $\lambda^{\delta-\eps} M \leq M$ if $\lambda<1$.

\subsection*{Local improvement of oscillation.}
The following theorem is the key result that leads to H\"older continuity in \cite{CV2007}.

\begin{thm} \label{t:improvementofoscilation}
Let $\theta$ be a solution to
\begin{align*}
\dv z^\eps \grad \theta &= 0 \qquad \text{where } z > 0 \\
\partial_t \theta(x,0,t) + w \cdot \grad \theta(x,0,t) +  \lim_{z \to 0} z^\eps \partial_z \theta(x,z,t) &= 0
\end{align*}
for an arbitrary divergence free vector field $w$ such that
\[ \norm{w}_{L^\infty([0,1],L^{2n/\alpha}(B_1))} \leq K. \]

Then \[ \osc_{\Q_{1/2}} \theta \leq (1-\eta) \osc_{\Q_1} \theta \]
for some $\eta>0$ depending only on $K$, $\eps$ and dimension (dimension is two in the quasi-geostrophic equation case).
\end{thm}

The proof of Theorem \ref{t:improvementofoscilation} was given in \cite{CV2007} for the case $\alpha = 1$. It relies only on a local energy inequality and De Giorgi's oscillation lemma. We prove both things in the appendix, so that the proof of Theorem \ref{t:improvementofoscilation} generalizes to smaller values of $\alpha$.

In \cite{CV2007}, the estimate in $L^\infty([0,1],L^{2n/\alpha}(B_1))$ was replaced by an estimate in $L^\infty(BMO)$ plus a control on the mean. Their assumption reads \[ \norm{w}_{L^\infty([0,1],BMO(\R^n))} + \sup_{[0,1]} \abs{\int_{B_1} w(x,t) \dx} \leq K.\]
This was done because the $L^{2n/\alpha}(B_1)$ norm is not invariant by the scaling of the equation. Since in this paper we will deal with scaling in a somewhat different way, we keep the sharp assumption from the proof, in $L^\infty([0,1],L^{2n/\alpha}(B_1))$.

The value of $\eta$ does depend on $\eps$. In particular it degenerates as $\eps \to 1$ (or equivalently as $\alpha = 1-\eps$ goes to zero). However, since in this paper we are interested only in the case of $\eps$ small, we can consider $\eta$ to be independent of $\eps$ (say for $\eps \in [0,1/2]$).
\section{Proofs}

In this section we provide the proofs of the main theorems \ref{t:main} and \ref{t:main2}.

\begin{lemma} \label{l:tails}
If $\theta$ is a solution of \eqref{e:equation}, then for any  $t>0$ we have the estimate
\begin{equation}
\int_{\R^2 \setminus B_1} \frac{|\theta(x,t)|}{|x|^2} \dx \leq C \norm{\theta_0}_{L^2}
\end{equation}

For any $t>1$, we have the improved estimate
\begin{equation}
\label{e:tails}
\int_{\R^2 \setminus B_1} \frac{|\theta(x,t)|}{|x|^2} \dx \leq C (1+\log t)  t^{-\alpha} \norm{\theta_0}_{L^2}
\end{equation}
\end{lemma}

\begin{proof}
For any $R>1$, we split the integral and use H\"older's inequality
\begin{align*}
\int_{\R^2 \setminus B_1} \frac{|\theta(x,t)|}{|x|^2} \dx &= \int_{B_R \setminus B_1} \frac{|\theta(x,t)|}{|x|^2} \dx + \int_{\R^2 \setminus B_R} \frac{|\theta(x,t)|}{|x|^2} \dx \\
&\leq \norm{\theta}_{L^\infty} \log R + \frac{C}{R} \norm{\theta}_{L^2}
\end{align*}

The first estimate follows if we pick $R=1$. Since the estimate holds for any $R$, when $t>1$ we choose $R = t^\alpha$. Using Theorem \ref{t:l2tolinf} and Theorem \ref{t:l2decreasing} we get,
\begin{align*}
\int_{\R^2 \setminus B_1} \frac{|\theta(x,t)|}{|x|^2} \dx &\leq \alpha \norm{\theta}_{L^\infty} \log t + C t^{-\alpha} \norm{\theta}_{L^2} \\
&\leq C (1+\log t)  t^{-\alpha} \norm{\theta_0}_{L^2}
\end{align*}
which finishes the proof.
\end{proof}

\begin{proof} [Proof of Theorem \ref{t:main}]
We will prove that $\theta$ is H\"older continuous at the point $(0,T)$. There is nothing special about $x=0$, so the proof implies the result of the theorem.

Let us choose some $t_0 < T$ (for example $t_0 = T/1000$), we have $\norm{\theta(-,t_0)}_{L^\infty(\R^3)} \leq C\norm{\theta_0}_{L^2}$ by Theorem \ref{t:l2tolinf}. Moreover, from Lemma \ref{l:tails}
\begin{equation} \label{e:outsideTheUnitBall}
\int_{\R^2 \setminus B_1} \frac{|\theta(x,t_0)|}{|x|^2} \dx \leq C \norm{\theta_0}_{L^2}.
\end{equation}

We normalize the problem in the following way. We consider \[\tilde \theta(x,z,t) = \frac{\theta(x/(T-T_0)^{\frac 1 \alpha},z/(T-T_0)^{\frac 1 \alpha},(t-t_0)/(T-T_0))}{C\norm{\theta_0}_{L^2}},\] so that $|\tilde \theta| \leq 1$ in $\R^2 \times [0,+\infty) \times [0,1]$,
\begin{equation} \label{e:tail}
\int_{\R^2 \setminus B_1} \frac{|\tilde \theta(x,t)|}{|x|^2} \dx \leq 1
\end{equation}
for $t \in [0,1]$, and $\tilde \theta$ solves \eqref{e:normalized1} and \eqref{e:normalized2}. Note that the constant $M$ depends only on $\norm{\theta(-,t_0)}_{L^\infty}$ and the right hand side of \eqref{e:outsideTheUnitBall} which are controlled by the $L^2$ norm of the original initial condition.

We stress that all estimates in the rest of this proof depend only on $\norm{\theta}_{L^\infty([t_0,+\infty),\R^n)}$ and the right hand side in \eqref{e:outsideTheUnitBall}.

From now on we will abuse notation by omitting the tilde in $\tilde \theta$ and we write just $\theta$. We assume $\norm{\theta(-,t)}_{L^\infty}\leq 1$ and \eqref{e:tail} for all $t \in [0,1]$.

We write \[ \theta_r(x,z,t) = \frac{1}{r^\delta} \theta \left( r x, r z , 1 - r^\alpha (1-t) \right). \]

Assuming that $\delta \ll 1$, we will show that $\osc_{\Q_r} \theta \leq C r^\delta$ for any $r<1$, obtaining H\"older continuity at the point $(0,1)$ (and by scaling and translation, at any point $(x,T)$ in the original equation). This is equivalent of saying that $\osc_{\Q_1} \theta_r \leq C$ for any $r<1$.

We will find a $1>\rho>0$ such that for $r_k = \rho^k$, $\osc_{\Q_1} \theta_{r_k} \leq 1$, and the result clearly follows.

We prove that $\osc_{\Q_1} \theta_{r_k} \leq 1$ by a usual iterative procedure, but since the equation is nonlocal, we must carry on some extra information in the iteration. In this case, the first step in the iterative process is a little bit different from the successive steps. We explain them separately to avoid confusion.

We stress that we need to choose $\rho>0$ and $\delta>0$ small. Then for $0 < \eps \leq \delta$ the theorem will apply. The choice of $\rho$ and $\delta$ must be made carefully. When we write a universal constant $C$ in this proof, we mean a constant that does not depend on $\rho$ or $\delta$.

\noindent \textbf{Step 1}

We start with $|\theta|\leq 1$ in $\R^2 \times [0,+\infty) \times [0,1]$ and $w = R^\bot \theta = (-R_2 \theta, R_1 \theta) \in L^\infty([0,1],BMO)$.

We also know \eqref{e:tail}, which tells us that the contribution of the tails in the integral representation of $R_i \theta$ are bounded. Equivalently, that $R_i (\theta (1-\chi_{B_2})) \in L^\infty$. On the other hand $R_i (\theta \chi_{B_2}) \in L^p$ for any $p<+\infty$ since $\theta \chi_{B_2}$ is bounded and compactly supported. Thus, for any $p<+\infty$, there is a constant $C$ such that $\sup_{t \in [0,1]} \norm{R_i \theta(-,t)}_{L^p(B_1)} \leq C$. In particular this estimate holds for $p = n/\alpha$ and we can apply Theorem \ref{t:improvementofoscilation} to get
\[ \osc_{\Q_{1/2}} \theta \leq 2 - 2\eta \]

Before rescaling $\theta$ to prepare for the next iterative step, we perform a small change of variables to \emph{follow the flow}. This is the key to make the iteration scheme succeed.

We write $w = w_1 + w_2$ where $w_2$ is given by the truncated integral
\[ w_2(x,t) = c \int_{\R^2 \setminus B_2} \frac{ \theta(y) (y-x)^\bot }{ |y-x|^{3} } \dd y \]

Note that $w_2$ is a continuous function in $x$. Let $V: [0,1] \to \R^2$ be a solution to the following ODE
\begin{align*}
V(1) &= 0 \\
\dot V(t) &= M w_2(V(t),t)
\end{align*}

We define $\tilde \theta(x,y,t) = \theta(x+V(t),y,t)$ and verify that $\tilde \theta$ satisfies the equation
\begin{align*}
\dv z^\eps \grad \tilde \theta &= 0 \qquad \text{where } z > 0 \\
\partial_t \tilde \theta(x,0,t) + M (w(x,t) - w_2(V(t),t)) \cdot \grad \tilde \theta(x,0,t) +  \lim_{z \to 0} z^\eps \partial_z \tilde \theta(x,z,t) &= 0
\end{align*}

From \eqref{e:tail}, we get that $|w_2| < L$ for some universal constant $L$. If we choose $\rho$ such that 
\begin{equation} \label{e:rho1}
L \rho^\alpha + \rho \leq 1/2 ,
\end{equation} then $(x+V(t),y,t) \in \Q_{1/2}$ if $(x,y,t) \in \Q_\rho$.

Now we rescale. For $m = (\sup_{\Q_{1/2}} \theta + \inf_{\Q_{1/2}} \theta)/2$, let 
\[ \theta_1(x,y,t) = (\tilde \theta - m)_\rho. \]

\begin{align*}
\dv z^\eps \grad \theta_1 &= 0 \qquad \text{where } z > 0 \\
\partial_t \theta_1(x,0,t) + r^{\delta-\eps} M (w(x,t) - \overline w(t)) \cdot \grad \theta_1(x,0,t) +  \lim_{z \to 0} z^\eps \partial_z \theta_1(x,z,t) &= 0
\end{align*}

where 
\[ \overline w(t) = c \int_{\R^2 \setminus B_{2/\rho}} \frac{ (\theta_1(y,t) - \theta_1(0,t)) y^\bot }{ |y|^3 } \dd y. \]

If $\delta$ is small enough so that $\rho^{-\delta} (1 - \eta) \leq 1$, then we will have $|\theta_1| \leq 1$ in $\Q_1$. Moreover $|\theta_1(x,t)| \leq \rho^{-\delta}$ where $|x|>1$.

We define $M_1 = r^{(\delta-\eps)k} M \leq M$. We are ready to move to the second step of the iteration.

\noindent \textbf{Step $k$ for $k>1$}

Assume that at the beginning of the $k$th step in the iteration we have a $\theta_k$ such that
\begin{align*}
\dv z^\eps \grad \theta_k &= 0 \qquad \text{where } z > 0 \\
\partial_t \theta_k(x,0,t) + M_k (w(x,t)-\overline w(t)) \cdot \grad \theta_k(x,0,t) +  \lim_{z \to 0} z^\eps \partial_z \theta_k(x,z,t) &= 0
\end{align*}

where $w = R^\bot \theta$ and
\[ \overline w = c \int_{\R^2 \setminus B_{2/\rho}} \frac{ \theta_k(y) y^\bot }{ |y|^3 } \dd y. \]

Moreover $|\theta_1| \leq 1$ in $\Q_1$ and $\theta_k(x,t) \leq 2 |x|^{2\delta}$  where $|x|>1$. Recall that $M_k \leq M$.

Let us write $w - \overline w = w_1 + w_2 + w_3$ where 
\begin{align}
w_1(x,t) &= c \int_{B_2} \frac{ \theta_k(y,t) (y-x)^\bot }{ |y-x|^3 } \dd y \\
w_2(x,t) &= c \int_{B_{2/\rho} \setminus B_2} \frac{ \theta_k(y,t) (y-x)^\bot }{ |y-x|^3 } \dd y \\
w_3(x,t) &= c \int_{\R^2 \setminus B_{2/\rho}} \theta_k(y,t) \left( \frac{(y-x)^\bot }{ |y-x|^3 } - \frac{y^\bot }{ |y|^3 } \right) \dd y
\end{align}

Let us analyze each component $w_i$. Since we are choosing $\delta$ small enough, we can and will assume $\rho^{-\delta} < 3/2 < 2$.

For estimating $w_1$, we notice that we are integrating a on a given compact domain $B_2$. Modulo a lower order correction, this is the same as applying a Riesz transform to a function with compact support. Therefore, from the $L^\infty$ estimate of $\theta$, we can apply classical Calderon-Zygmund estimates, we obtain that $w_1$ is in $L^\infty([0,1],L^p(B_1))$ for any $p \in (1,\infty)$. In particular for $p = 2n / \alpha$, and its norm (for this particular $p$) is less than a universal constant $K$ (in this case it does not depend even on $\rho$).

Both $w_2$ and $w_3$ are bounded. We will estimate their $L^\infty$ norms in $\Q_1$, which is stronger than the norms in $L^\infty([0,1],L^{2n/\alpha}(B_1))$.
\begin{align*}
|w_2(x,t)| &\leq c \int_{B_{2/\rho} \setminus B_2} \frac{ 2^{1+2\delta} \rho^{-2\delta} }{ |y-x|^2 } \dd y  && \text{using that } |\theta_k(y,t)| \leq 2^{1+2\delta} \rho^{-2\delta} \text{ if } y \in B_{2/\rho}\\
&\leq -C \log \rho
\end{align*}

\begin{align*}
|w_3(x,t)| &\leq c \int_{\R^2 \setminus B_{2/\rho}} 2 |y|^{2\delta} \abs{ \frac{(y-x)^\bot }{ |y-x|^3 } - \frac{y^\bot }{ |y|^3 } } \dd y && \text{using that } |\theta_k(y,t)| \leq 2|y|^{2\delta} \text{ where } |x|>1 \\
&\leq c \int_{\R^2 \setminus B_{2/\rho}} 2 |y|^{2\delta} \frac{C|x|}{|y|^3} \dd y && \text{where } |x|\leq 1\\
&\leq C \rho
\end{align*}

Since $w_1+w_2+w_3 \in L^\infty([0,1],L^{2n/\alpha}(B_1))$, we can apply Theorem \ref{t:improvementofoscilation} again to obtain $\osc_{\Q_{1/2}} \theta \leq 2-2\eta$ (where $\eta$ depends on $\rho$).

Since $\overline w_2$ is continuous, we can solve the equation as before
\begin{align*}
V(1) &= 0 \\
\dot V(t) &= M_k ( w_2(V(t),t) + w_3(V(t),t) )
\end{align*}

Note that from the estimates above for $|w_1|$ and $|w_2|$ we have $|\dot V(t)| \leq -C \log \rho + C \rho$. Therefore $|V(t)| \leq -C \rho^\alpha \log \rho + C \rho^{1+\alpha}$ if $t \in [(1 - \rho^\alpha),1]$.

We choose $\rho$ small such that 
\begin{equation} \label{e:rho2}
-C \rho^\alpha \log \rho + C \rho^{1+\alpha} + \rho \leq 1/2
\end{equation}
so as to make sure that $(x+V(t),y,t) \in \Q_{1/2}$ if $(x,y,t) \in \Q_\rho$. Note that there is no circular dependence of constants since the constants $C$ above are universal.

We continue as in step 1. We define $\tilde \theta_k(x,y,t) = \theta_k(x+V(t),y,t)$ and verify that  $\tilde \theta_k$ satisfies the equation
\begin{align*}
\dv z^\eps \grad \tilde \theta_k &= 0 \qquad \text{where } z > 0 \\
\partial_t \tilde \theta_k(x,0,t) + M_k \tilde w \cdot \grad \tilde \theta_k(x,0,t) +  \lim_{z \to 0} z^\eps \partial_z \tilde \theta_k(x,z,t) &= 0
\end{align*}
with 
\begin{align*}
 \tilde w(x,t) &= w(x+V(t),t) - \overline w(t) - \frac{\dot V(t)} {M_k} \\ &= w_1(x+V(t),t)+w_2(x+V(t),t)+w_3(x+V(t),t)-w_2(V(t),t)-w_3(V(t),t) \\
 &= c \left( \int_{\R^n} \frac{ \tilde \theta_k(y,t) (y-x)^\bot }{ |y-x|^3 } \dd y - \int_{\R^n \setminus B_2} \frac{ \tilde \theta_k(y,t) y^\bot }{ |y|^3 } \dd y \right).
\end{align*}

Now we rescale. Since $\osc_{\Q_\rho} \tilde \theta \leq 2-2\eta \leq 2 \rho^\delta$. Let us pick $m \in [-1+\rho^\delta,1-\rho^\delta]$ such that $|\tilde \theta - m| \leq \rho^\delta$ in $\Q_\rho$ (typically $m = (\sup_{\Q_{1/2}} \theta_k + \inf_{\Q_{1/2}} \theta_k)/2$). Let 
\[ \theta_{k+1}(x,y,t) = (\tilde \theta - m)_\rho. \]

\begin{align*}
\dv z^\eps \grad \theta_{k+1} &= 0 \qquad \text{where } z > 0 \\
\partial_t \theta_{k+1}(x,0,t) + M_{k+1} (w(x,t)-\overline w(t)) \cdot \grad \theta_{k+1}(x,0,t) +  \lim_{z \to 0} z^\eps \partial_z \theta_{k+1}(x,z,t) &= 0
\end{align*}

where $M_{k+1} = r^{\delta-\eps} M_k \leq M$, $w = R^\bot \theta$ and 
\[ \overline w(t) = c \int_{\R^2 \setminus B_{2/\rho}} \frac{ \theta_{k+1}(y) y^\bot }{ |y|^3 } \dd y. \]

Then we will have $|\theta_{k+1}| \leq 1$ in $\Q_1$. To make sure we obtain our desired estimates when $|x|>1$ we must make some computations which follow.
\begin{align*}
|\theta_{k+1}(x,t)| &\leq \rho^{-\delta} |\tilde \theta(\rho x, \rho^\alpha t) - m| \\
&\leq \rho^{-\delta} (m + |\theta_k(\rho x + V(\rho^\alpha t)), \rho^\alpha t)|) \\
&\leq \begin{cases}
\rho^{-\delta} (2-\rho^\delta)  & \text{if } |x| \leq \frac{1}{2\rho} \\
\rho^{-\delta} (1-\rho^\delta + \rho^{-2\delta} (\rho |x| + 1/2)^{2 \delta}) & \text{if } |x| > \frac{1}{2\rho}
\end{cases}
\end{align*}

In case $1 \leq |x| \leq \frac{1}{2\rho}$ we have $|\theta_{k+1}(x,t)| \leq \rho^{-\delta} (2-\rho^\delta) \leq 2 \leq 2|x|^{2\delta}$ since $\rho^\delta$ was chosen larger than $2/3$.

In case $|x| \geq \frac{1}{2\rho}$, we have 
\begin{align*}
|\theta_{k+1}(x,t)| &\leq \rho^{-\delta} (1-\rho^\delta + 2 (\rho |x| + 1/2)^{2\delta}) \\
&\leq \rho^{-\delta} - 1 + 2 \rho^{-\delta} (2\rho |x|)^{2\delta} \\
&\leq \rho^{-\delta} - 1 + 2 \rho^\delta 2^{2\delta} |x|^{2\delta} \\
&\leq 2 |x|^{2\delta} \left(\frac{\rho^{-\delta}}{2|x|^{2\delta}} - \frac{1}{2|x|^{2\delta}} + 2^{2\delta} \rho^\delta \right) \\
&\leq 2 |x|^{2\delta} \left(\frac{(4\rho)^{\delta}}{2} - \frac{4^\delta \rho^{2\delta}}{2} + (4\rho)^\delta \right) \\
&\leq 2 |x|^{2\delta} (4\rho)^{\delta} \left(\frac{3}{2} - \frac{\rho^{\delta}}{2} \right) < 2 |x|^{2\delta}
\end{align*}

where the last inequality holds if $\rho \leq 1/16$, since then we would have
\[ (4\rho)^{\delta} \left(\frac{3}{2} - \frac{\rho^{\delta}}{2} \right) < \rho^{\delta/2} \left(\frac{3}{2} - \frac{\rho^{\delta}}{2} \right) \]
which is less than $1$ for any $\delta>0$ (the polynomial $x(3/2 -x^3/2)$ has a maximum at $x=1$).

Therefore in every case we obtained $|\theta_{k+1}| \leq 1$ in $\Q_1$ and $|\theta_{k+1}| \leq 2|x|^{2\delta}$ when $|x|>1$. We finish step $k$ and are ready for the next step in the iteration.

We stress that there is no circular dependence in the choice of constants. The constant $\rho$ is the first one which has to be chosen. It must satisfy three inequalities.
\begin{itemize}
\item $L \rho^\alpha + \rho \leq 1/2$ for \eqref{e:rho1} in the first step.
\item $-C \rho^\alpha \log \rho + C \rho^{1+\alpha} + \rho \leq 1/2$ for \eqref{e:rho2}.
\item $\rho < 1/16$ in order to make the very last inequality work.
\end{itemize}
All the constants above depend only $M$ and \eqref{e:tail}, which both depend only on $\norm{\theta_0}_{L^2}$.

Once we have $\rho$, the value of $\eta$ follows from applying theorem \ref{t:improvementofoscilation}. So $\eta$ depends on the initial choice of $\rho$. Once we have $\eta$ and $\rho$, we choose $\delta$ so that $\rho^\delta \geq (1-\eta)$ and $\rho^{-\delta} \leq 2$.
\end{proof}

\begin{proof} [Proof of Theorem \ref{t:main2}]
Using Theorem \ref{t:l2tolinf}, there is a $t_0$ depending only on $\norm{\theta_0}_{L^2}$ such that $\norm{\theta(-,t_0)}_{L^\infty} \leq 1$. We can also pick $t_0$ large such that the right hand side in Lemma \ref{l:tails} is smaller than $1$. At that point, we are already in the normalized situation of the proof of Theorem \ref{t:main}, the choice of all constants from that point on does not depend on $\norm{\theta_0}_{L^2}$.

We consider the function $\theta$ starting at this $t_0$ and we apply theorem \ref{t:main} with $T=1$, $M=1$ and the right hand side in \eqref{e:tail} equal to $1$. Then if $\eps \leq \delta$ for $\delta$ small enough, we will have $\theta \in C^\delta$ at any time $t \geq t_0 + 1$. Further $C^\infty$ regularity follows from \cite{constantin2007rhc}.
\end{proof}

\section*{Appendix}

In this appendix we present the local energy inequality and a weighted version of De Giorgi's isoperimetrical lemma so that we can reproduce the proof in \cite{CV2007} of Theorem \ref{t:improvementofoscilation}. The proofs use essentially the same ideas as in \cite{CV2007}. The main modification is that we need to use the weight $z^\eps$ in the upper half space, so that the Dirichlet to Neumann map for harmonic functions corresponds to the fractional laplacian $\fl{\alpha}$ (See \cite{CS2007}).

For the local energy inequality, we can deal with a more general equation
\begin{equation} \label{e:cveq}
\begin{aligned}
\dv z^\eps \grad \theta &= 0 \qquad \text{where } z > 0 \\
\partial_t \theta(x,0,t) + w \cdot \grad \theta(x,0,t) +  \lim_{z \to 0} z^\eps \partial_z \theta(x,z,t) &= 0
\end{aligned}
\end{equation}
where $w$ is a fixed divergence free vector field in $\R^n$ and $\theta : \R^n \times [0,\infty) \times [0,\infty) \to \R$. For proving Theorem \ref{t:improvementofoscilation} we do not need to use the relation between $w$ and $\theta$, and the dimension $n$ is arbitrary.

Note that after restricting $\theta$ to $z=0$, \eqref{e:cveq} is equivalent to
\begin{equation}
\partial_t \theta(x,t) + w \cdot \grad \theta(x,t) +  \fl{\alpha} \theta(x,t) = 0
\end{equation}

\begin{prop}[Local energy inequality]
Let $t_1 < t_2$ and let $\theta \in L^\infty(t_1,t_2; L^2(\R^n))$ with $\fl{\alpha} \theta \in L^2((t_1,t_2) \times \R^n)$, be a solution to \eqref{e:equation} with velocity $w$ satisfying:
\[ \norm{w}_{L^\infty(t_1,t_2;L^{2n/\alpha}(B_2))} \leq C \]
Then there exists a constant $C_1$ (depending only on $C$) such that for every $t \in (t_1,t_2)$ and cut-off function $\eta$ compactly supported in $B_2^*$:
\begin{equation}\label{e:e1}\begin{aligned}
 \int_{t_1}^{t_2} \int_{B_2^*} z^\eps |\grad(\eta[\theta^*]_+)|^2 \dx \dd z \dd t &+ \int_{B_2} (\eta [\theta]^+)^2 (x,t_2) \dx \\
 &\leq \int_{B_2} (\eta [\theta]_+)^2(x,t_1) \dx + C_1 \int_{t_1}^{t_2} \int_{B_2} (|\grad \eta| [\theta]_+)^2 \dx \dd t \\
 & \phantom{\leq \int_{B_2} (\eta [\theta]_+)^2(x,t_1) \dx} + C_1 \int_{t_1}^{t_2} \int_{B_2^*} z^\eps (|\grad \eta| [\theta]_+)^2 \dx \dd z \dd t .
\end{aligned}\end{equation}
\end{prop}
 
Note that the only difference with the corresponding estimate in \cite{CV2007} is the factor $z^\eps$ in every integral involving the extension to $z>0$. This is a straight forward modification following \cite{CS2007}. This only modification applies along the proof.

Note also that the BMO norm plus an estimate on the mean is stronger than $L^{2n/\alpha}$, so in particular the estimate holds if $w \in L^\infty(BMO)$ and the mean of $w$ in $B_2$ is also bounded.

\begin{proof}
We have for every $t \in (t_1,t_2)$:
\[ \begin{aligned}
 0 &= \int_0^\infty \int_{\R^n} \eta^2 [\theta]_+ \dv (z^\eps \grad \theta) \dx \dd z \\
 &= \int_0^\infty \int_{\R^n} -z^\eps |\grad(\eta [\theta]_+)|^2 + z^\eps |\grad \eta|^2 [\theta]_+^2 \dx \dd z + \int_{\R^n} \eta^2 [\theta]_+ \fl{\alpha} \theta \dx
\end{aligned} \]
where the characterization of $\fl{\alpha} \theta$ as the Dirichlet to Neumann operator from \cite{CS2007} was used.

As in \cite{CV2007}, we use the equation \eqref{e:equation} which leads to
\[ \begin{aligned}
&\int_{t_1}^{t_2} \int_0^\infty \int_{\R^n} z^\eps |\grad(\eta [\theta]_+)|^2 \dx \dd z \dd t + \int_{\R^n} \eta^2 \frac{[\theta]_+^2 (t_2)}{2} \dx \\
&\qquad \leq \int_{\R^n} \eta^2 \frac{[\theta]_+^2 (t_1)}{2} \dx + \int_{t_1}^{t_2} \int_0^\infty \int_{\R^n} z^\eps |\grad \eta|^2 [\theta]_+^2 \dx \dd z \dd t + \abs{ \int_{t_1}^{t_2} \int_{\R^n} \eta \grad \eta \cdot w [\theta]_+^2 \dx \dd t}.
\end{aligned} 
\]

To dominate the last term, we use Sobolev embedding and the variational characterization of the equation $\dv z^\eps \grad U = 0$.

\[
\begin{aligned}
\norm{\eta [\theta]_+}^2_{L^{\frac{2n}{n-2\alpha}}(\R^n)} &\leq C \norm{\eta [\theta]_+}^2_{H^{\alpha/2}} = C \int_{\R^n} (\eta \theta_+) \fl{\alpha} \theta_+ \dx \\
&= \min_{v(x,0,t) = \eta(x,0,t) \theta(x,0,t)} \int_{\R^n \times (0,+\infty)} z^\eps |\grad v|^2 \dx \dd z \\
&\leq \int_{B^*_2} z^\eps |\grad (\eta \theta_+)|^2 \dx \dd z
\end{aligned}
\]
Recall that $\eta$ is supported inside $B_2$. Now we continue in the standard way as in \cite{CV2007}. For some small $\tilde \eps$, we write
\[ 
\abs{\int_{t_1}^{t_2} \int_{\R^n} \grad \eta^2 \cdot w \frac{\theta_+^2}{2} \dx \dd t }
\leq \tilde \eps \int_{t_1}^{t_2} \norm{\eta \theta_+}_{L^{\frac{2n}{n-\alpha}}} \dd t + \frac{C}{\tilde \eps} \int_{t_1}^{t_2} \norm{ \grad \eta \cdot w \theta_+ }_{L^{\frac{2n}{n+\alpha}}}^2 \dd t
\]

The first term is absorbed by the left hand side in \eqref{e:e1}. The second is bounded using H\"older's inequality:
\[ \frac{C}{\tilde \eps} \int_{t_1}^{t_2} \norm{ \grad \eta \cdot w \theta_+ }_{L^{\frac{2n}{n+\alpha}}}^2 \dd t \leq \frac{C}{\tilde \eps} \norm{w}_{L^\infty(t_1,t_2;L^{2n/\alpha}(B_2))}
\int_{t_1}^{t_2} \int_{B_2} | \grad \eta \ \theta_+ |^2 \dx \dd t\]
which finishes the proof.
\end{proof}

Now we show the De Giorgi isoperimetrical lemma with the weight $z^\eps$. This is a property of $H^1$ functions independently of the equation.

\begin{prop}[De Giorgi isoperimetrical lemma]
Let $w$ be a function in $H^1(B^*_1, z^\eps)$, we have the estimate
\begin{equation}\label{e:dgisop}
\left( \int_{\{w \leq 0\}} z^\eps \dd X \right) \left( \int_{\{w \geq 1\}} z^\eps \dd X \right) \leq C \left(\int_{\{0 < w < 1\}} z^\eps \dd X \right)^{1/2} \left(\int_{B_1^*} |\grad w| z^\eps \dd X \right)^{1/2}
\end{equation}
(Recall the notation $X = (X',X_{n+1}) = (x,z)$ with $X \in \R^{n+1}$ and $x =X' \in \R^n$).
\end{prop}

This estimate can also be written as
\[ |\{w \leq 0\}| \ |\{w \geq 1\}| \leq C |\{0< w <1\}|^{1/2} \norm{w}_{\dot H^1(z^\eps)} \]
where the measures of the sets are computed with respect to the weight $z^\eps$.

Note that \eqref{e:dgisop} is not scale invariant. If we replace $B_1^*$ by $B_r^*$, the constant $C$ would depend on $r$.

\begin{proof}
We can consider $\tilde w = \max(0,\min(1,w))$, so it is no loss of generality to assume $w(X) \in [0,1]$ for every $X$, so that $|\grad w|=0$ a.e. in $\{w\geq 0\}$ and $\{w \leq 0\}$.

Let $X$ be a point such that $w(X) = 0$ and $Y$ be such that $w(Y) = 1$. Let $\theta = \frac{Y-X}{|Y-X|}$, we compute
\[ Y_{n+1}^\eps \leq Y_{n+1}^\eps \int_0^{|Y-X|} |\grad w(X + t \theta)| \dd t \]

For a fixed value of $X$ we write
\begin{align}
Z &= X + t \theta \\
Y &= X + m \theta \\
\dd t \dd Y &= \frac{m^{n-1}}{t^{n-1}} \dd m \dd Z 
\end{align}

In order to estimate the second factor in the left hand side, we integrate in $Y$.
\begin{align*}
 \left( \int_{\{w(Y) = 1\}} Y_{n+1}^\eps \dd Y \right) &\leq 
\int_{\{w(Y) = 1\}} \int_0^{|Y-X|} Y_{n+1}^\eps |\grad w(X + t \theta)| \dd t \dd Y \\
&\leq \int_{\{0 < w(Z) < 1\}} \frac{|\grad w(Z)|}{|X-Z|^n} \int_{m_0}^{m1} m^{n-1} \left[ X + m \theta \right]_{n+1}^\eps \dd m \dd Z \\
&\leq C \int_{\{0 < w(Z) < 1\}} \frac{|\grad w(Z)|}{|X-Z|^n} [Y^*]_{n+1}^\eps \dd Z
\end{align*}
where $Y^*$ is the point on the line $X + m\theta$ such that $Y_{n+1}$ is maximum. Note that $\theta = \frac{Z-X}{|Z-X|}$.

Now we integrate in $X$.
\begin{align*}
 \left( \int_{\{w(X) = 0\}} X_{n+1}^\eps \dd X \right) & \left( \int_{\{w(Y) = 1\}} Y_{n+1}^\eps \dd Y \right) \\
  &\leq C \int_{\{w(X)=0\}} \int_{\{0 < w(Z) < 1\}} \frac{|\grad w(Z)|}{|X-Z|^n} [Y^*]_{n+1}^\eps X_{n+1}^\eps \dd Z \dd X
\end{align*}

The point $Y^*$ depends on $X$ and $Z$. In every case, $Z$ is on the line segment joining $X$ and $Y^*$, so either $X_{n+1} \leq Z_{n+1}$ or $Y^*_{n+1} \leq Z_{n+1}$. On the other hand $\max(X_{n+1},Y^*_{n+1}) \leq 1$ since $X,Y^* \in B_1^*$. Thus $[Y^*]_{n+1}^\eps X_{n+1}^\eps \leq Z_{n+1}^\eps$. Therefore
\begin{align*}
 \left( \int_{\{w(X) = 0\}} X_{n+1}^\eps \dd X \right) & \left( \int_{\{w(Y) = 1\}} Y_{n+1}^\eps \dd Y \right) \\ 
  &\leq C \int_{\{w(X)=0\}} \int_{\{0 < w(Z) < 1\}} \frac{|\grad w(Z)|}{|X-Z|^n} Z_{n+1}^\eps \dd Z \dd X\\
  &\leq C \int_{\{0 < w(Z) < 1\}} |\grad w(Z)| Z_{n+1}^\eps \dd Z \\
  &\leq C \left(\int_{\{0 < w < 1\}} Z_{n+1}^\eps \dd Z \right)^{1/2} \left(\int_{B_1^*} |\grad w(Z)| Z_{n+1}^\eps \dd Z \right)^{1/2}
\end{align*}
The last inequality follows by Cauchy-Schwartz since the support of $\grad w$ is included in $\{0<w<1\}$ (we are assuming $0 \leq w \leq 1$ in $B_1^*$). This finishes the proof.
\end{proof}

\bibliographystyle{plain}   
\bibliography{supercriticalqg}             
\index{Bibliography@\emph{Bibliography}}%

\end{document}